\documentclass[11pt, reqno]{amsart}
\usepackage[a4paper, margin=3 cm]{geometry}
\usepackage{amsfonts}
\usepackage{amsmath}
\usepackage{amssymb}
\usepackage{amsthm}
\usepackage{enumerate}
\usepackage{mathdots}
\usepackage{mathrsfs}
\usepackage{amsbsy}
\usepackage[all]{xy}
\usepackage{bm}
\usepackage{tikz}
\usepackage{array}
\usepackage{float}
\usepackage{xcolor}
\usepackage{hhline}
\usepackage{mathtools}
\usepackage[colorlinks=true, linkcolor=blue, citecolor=blue, urlcolor=blue]{hyperref}
\usepackage{caption}
\usepackage{subcaption}
\usepackage{xcolor}
\usepackage{adjustbox}
\usepackage{mathtools}
\usepackage{url}
\usepackage[noadjust]{cite}
\usepackage{tabu}
\usepackage{diagbox}
\usepackage{tikz}
\usepackage{bbm}
\usepackage{booktabs}
\usepackage[noabbrev,capitalise]{cleveref}

\newtheorem{theorem}{Theorem}[section]
\newtheorem{lemma}[theorem]{Lemma}
\newtheorem{proposition}[theorem]{Proposition}
\newtheorem{definition}[theorem]{Definition}
\newtheorem{cor}[theorem]{Corollary}

\newtheorem*{rmk*}{Remark}
\newtheorem*{lem*}{Lemma}
\newtheorem*{claim*}{Claim}

\newtheorem{conjecture}[theorem]{Conjecture}

\theoremstyle{definition}

\newcommand{\md}[1]{\ensuremath{(\mbox{mod}\, #1)}}
\newcommand{\mdsub}[1]{\ensuremath{(\mbox{\scriptsize mod}\, #1)}}
\DeclareMathOperator{\Span}{span}

\DeclareMathOperator{\ML}{\operatorname{ML}}

\newcommand{\To}{\mathbf{T}}

\allowdisplaybreaks

\newenvironment{enumerate*}%
  {\begin{enumerate}[(I)]%
    \setlength{\itemsep}{10pt}%
    \setlength{\parskip}{0pt}}%
  {\end{enumerate}}

\usepackage{fancyhdr}

\title{Riesz products and the Lonely Runner Conjecture: \\ A wider gap of loneliness}
\author{Benjamin Bedert}
\thanks{bedert.benjamin@gmail.com\\The author gratefully acknowledges financial support from the European Research Council (ERC) Starting Grant “High Dimensional Probability and Combinatorics”, grant No. 101165900.}
\begin{document}

\begin{abstract}
The lonely runner conjecture of Wills and Cusick asserts that if $n$ runners with distinct constant speeds run around the unit circle, starting at a common time and place, then each runner will at some time be separated by a distance of at least $\frac{1}{n}$ from all other runners. A weaker lower bound of $\frac{1}{2n-2}$ follows from the so-called trivial union bound, and subsequent work upgraded this to bounds of the form $\frac{1}{2n}+\frac{c}{n^2}$ for various constants $c>0$. Tao strengthened this to $\frac{1}{2n}+\frac{(\log n)^{1-o(1)}}{n^2}$. In this paper, we obtain a polynomial improvement of the form $$\frac{1}{2n}+\frac{1}{n^{5/3+o(1)}}.$$
\end{abstract}

\maketitle
\tableofcontents
\section{Introduction}
The Lonely Runner Conjecture is an influential research question that was first raised by Wills \cite{Wills1967} in the context of Diophantine approximation, and independently by Cusick \cite{Cusick1973} in the context of view-obstruction problems. Its name comes from the following imaginative reformulation due to
Goddyn \cite{goddyn}. Consider $n$ runners moving around the unit circle
with distinct constant speeds, starting at a common time and place. Then the conjecture
states that each runner will at some time be separated by a distance of at least $\frac{1}{n}$ from all other runners. Let us identify the unit circle with $\To=\mathbf{R}/\mathbf{Z}$, and $\lVert \cdot\rVert_\mathbf{T}$ denotes the distance to the nearest integer. This allows us to formulate the Lonely Runner Conjecture as follows. \begin{conjecture}[Lonely Runner Conjecture]\label{conj:lonerunn0}
    Let $w_1,w_2,\dots,w_n\in\mathbf{N}$ be any $n$ distinct positive integers. For each $i\in[n]$, there exists a `time' $t\in\mathbf{R}$ such that $$\min_{j\neq i} \lVert t(w_i-w_j)\rVert_\mathbf{T}\geqslant \frac{1}{n}.$$
\end{conjecture} The original version of this conjecture asserts that such a conclusion holds whenever the $w_j$ are distinct real numbers. It is well-known however (see for example \cite[Section 4]{BohmanHolzmanKleitman2001}) that this seemingly more general version is equivalent to the integer version in Conjecture \ref{conj:lonerunn0}. As the problem is invariant under translations of the speed set $\{w_1, w_2, \dots, w_n\}$, we may without loss of generality subtract the speed of the runner under consideration, thereby reducing their speed to zero. Hence, Conjecture \ref{conj:lonerunn0} for $n+1$ runners is equivalent to proving that, given any $n$ distinct positive integers $v_1,v_2,\dots,v_n$, there is a `time' $t\in\mathbf{R}$ such that \begin{align*}
    \min_{j=1,\dots,n}\lVert t v_j\rVert_\mathbf{T}\geqslant \frac{1}{n+1}.
\end{align*}
Since all speeds $v_j$ are integers, it suffices to consider $t\in\mathbf{T}=\mathbf{R}/\mathbf{Z}$. This suggests defining the following quantity\footnote{This notation is taken from a paper of Kravitz \cite{kravitz}.} for any set $V=\{v_1,v_2,\dots,v_n\}$ of $n$ distinct positive integers:
\begin{align}
    \ML(V)=\ML(v_1,\dots,v_n)\vcentcolon=\max_{t\in \mathbf{T}}\left(\min_{j=1,\dots,n}\lVert tv_j\rVert_\mathbf{T}\right),
\end{align}
where `ML' stands for the maximum loneliness.
This leads us to the following succinct formulation.
\begin{conjecture}[Lonely Runner Conjecture]\label{conj:lonerunn1}
If $v_1,v_2,\dots,v_n$ are $n$ distinct positive integers, then $\ML(v_1,\dots,v_n)\geqslant \frac{1}{n+1}$.
\end{conjecture}
We remark that, if true, this lower bound would be optimal since a simple check confirms that $\ML(1,2,\dots,n)=\frac{1}{n+1}$. There is a vast literature on the Lonely Runner Conjecture, spanning a wide range of topics such as Diophantine approximation (e.g.~\cite{betke_wills1972,wills1965,Wills1967,wills1968}), view-obstruction/billiard trajectory problems (e.g.~\cite{Cusick1973,schoenberg1976, chen1994}), chromatic numbers
of distance graphs (e.g.~\cite{eggleton1985,chang1999,liu_zhu2004,liu2008,zhu2001}), and flows in matroids (e.g.~\cite{goddyn}). We refer the reader to the very recent survey of Perarnau and Serra \cite{serrasurvey} for a more comprehensive overview. 

\medskip

The principal lines of attack on the Lonely Runner Conjecture can be grouped into three categories.
\par\textbf{(I) Few runners:} Conjecture \ref{conj:lonerunn1} has been confirmed for $n \leqslant 7$ (corresponding to at most eight runners in Conjecture \ref{conj:lonerunn0}), see \cite{BohmanHolzmanKleitman2001, BarajasSerra2008, Renault2004, CusickPomerance1984, betke_wills1972, Rosenfeld2025LonelyRunner}.
\par\textbf{(II) Stronger assumptions on the speeds:} It has been established that Conjecture \ref{conj:lonerunn1} holds for sets $\{v_1,\dots,v_n\}$ satisfying one of various additional assumptions, such as being lacunary (e.g.~\cite{ruzsa_tuza_voigt2002,barajas_serra2009,pandey2009,dubickas2011,czerwinski2018}), or random-like (e.g.~ \cite{czerwinski2012_random}).
\par\textbf{(III) Increasing the `gap of loneliness':}
The starting point here is the classical observation that for any $n$ positive integers $v_1,\dots,v_n$, a trivial union bound shows that $$\ML(v_1,\dots,v_n)\geqslant \frac{1}{2n},$$  which gets within a factor of two of Conjecture \ref{conj:lonerunn1}. To see this, suppose for a contradiction that $\ML(v_1,\dots,v_n)=\delta< 1/2n$. This means that for every $t\in\To$, there exists some $j$ such that $tv_j\in I:=[-\delta,\delta]\subset \To$. However, the interval $I$ has measure $2\delta<1/n$ and hence, if we pick $t\in\mathbf{T}$ uniformly at random, then a union bound shows that $$\mathbb{P}(tv_j\in I\text{ for some  }j\in[n] )\leqslant \sum_{j=1}^n\mathbb{P}(tv_j\in I)=n\times2\delta<1,$$ giving the required contradiction. Surprisingly, for all large $n$, only small improvements over this crude union bound $\ML(V)\geqslant \frac{1}{2n}$ have been obtained. 
Chen~\cite{chen1994} showed that
\begin{equation*} 
  \ML(v_1,\dots,v_n) \geqslant\frac{1}{2n-1 + \frac{1}{2n-3}},
\end{equation*}
and Chen and Cusick~\cite{chen_cusick1999} improved this to
$\ML(v_1,\dots,v_n) \geqslant \frac{1}{2n-3}$ if $2n-3$ is prime. Pararnau and Serra~\cite{perarnau_serra2016} established the bound
\begin{equation}\label{eq:paraserr}
  \ML(v_1,\dots,v_n) \geqslant \frac{1}{2n-2+ o(1)}.
\end{equation}
Note that all of the bounds above are (essentially) of the form $\frac{1}{2n}+\frac{c}{n^2}$ for various constants $c$. Tao \cite{Tao2018LonelyRunner} was the first to show that one may replace this $c$ by a function $(\log n)^{1-o(1)}$ which grows with $n$, proving that
\begin{equation}\label{eq:tao}\ML(v_1,\dots,v_n)\geqslant \frac{1}{2n}+c'\frac{\log n}{n^2(\log\log n)^2}
\end{equation}
for an absolute constant $c'>0$.
Our contribution is the following improvement for the `gap of loneliness', strengthening the factor $(\log n)^{1-o(1)}$ to polynomial growth of order $n^{1/3-o(1)}$. 
\begin{theorem}\label{th:main}
    Let $V$ be a set of $n$ distinct positive integers. Then $$\ML(V)\geqslant \frac{1}{2n}+\frac{1}{n^{5/3+o(1)}}.$$
\end{theorem}
Along the way, we prove several intermediate results which may be of independent interest, such as Lemma \ref{lem:disslonerunn} which confirms the Lonely Runner Conjecture for every dissociated set (see Definition \ref{def:addidimensio}), and Proposition \ref{prop:largedimintro} which proves strong bounds for $\ML(V)$ whenever $V$ has rather large additive dimension (also see Definition \ref{def:addidimensio}).

\smallskip

\par\textbf{Acknowledgements:} The author would like to thank Thomas F. Bloom for several insightful discussions, and Noah Kravitz for helpful comments on earlier versions of this paper. 

\section{Overview of the proof}\label{sec:overview} 
We continue with an overview of the main ideas in our approach. So consider an arbitrary fixed collection $V=\{v_1,v_2,\dots,v_n\}\subset\mathbf{N}$ of $n$ distinct positive integers, and let $\delta:=\ML(v_1,\dots,v_n)$. Define the Bohr sets $$B(v_j;\delta):=\{t\in\mathbf{T}:\lVert tv_j\rVert_\To\leqslant \delta\}.$$ Observe that as $\delta=\ML(v_1,\dots,v_n)$, there must for every $t\in \To$ exist a $j\in[n]$ such that $\lVert tv_j\rVert_\To\leqslant \delta$, and hence the Bohr sets $B(v_j;\delta)$ cover all of $\To$:
\begin{equation}\label{eq:Bohrcover}
    \bigcup_{j=1}^nB(v_j;\delta)=\To. 
\end{equation}
Pararnau and Serra \cite{perarnau_serra2016} obtained their improvement \eqref{eq:paraserr} over the trivial bound for the gap of loneliness by analysing the rank 2 Bohr sets $$B(v_j,v_k;\delta):=B(v_j;\delta)\cap B(v_k;\delta).$$ Tao's approach further exploits control on rank 3 Bohr sets, showing that sets $V$ for which $\ML(V)$ is `close' to the trivial bound must be additively structured. He then uses a direct argument based on an analysis of the possible prime factors of such structured sets $V$ to see that $\ML(V)$ still exceeds the trivial union bound. However, Tao remarks (see the discussion following Theorem 1.2 in \cite{Tao2018LonelyRunner}) that his methods would require significantly more effort to allow one to improve the single power of $\log n$ in \eqref{eq:tao}. 

\medskip

The approach in this paper is rather different, and partially inspired by the structure vs.~randomness dichotomy based on additive dimension that appears in the author's recent work \cite{bedertsum-free} on sum-free sets. Recall that $\delta:=\ML(V)$. Define the indicator function $\phi:\mathbf{T}\to\{0,1\}:\phi(x)=\mathbf{1}_{[-\delta,\delta]}(x)$, and let
\begin{equation}\label{eq:Phidefi}
    \Phi_{V,\delta}(x):=\sum_{v\in V}\phi(vx).
\end{equation} We shall often simply write $\Phi(x)=\Phi_{V,\delta}(x)$, and it will be clear from context what $V$ is (and we always take $\delta=\ML(V)$). Observe that $\phi(vx)=\mathbf{1}_{B(v;\delta)}(x)$ and hence that \eqref{eq:Bohrcover} is the equivalent to the assertion that $\Phi(x)\geqslant 1$ for all $x\in \To$. The proof of the trivial union bound can therefore be rewritten as $$1\leqslant\int_\To \Phi(x)\,dx=\sum_{v\in V}\int_\To\phi(vx)\,dx=n\int_\To\phi(u)\,du=2\delta n,$$ so that $\ML(V)=\delta\geqslant \frac1{2n}$. The starting point of our approach is to exploit the stronger fact that $\Phi(x)\geqslant 1$ holds pointwise, by testing $\Phi$ against an arbitrary probability measure (rather than the uniform measure on $\To$ that is used in the trivial bound). We write $M(\To)$ for the set of regular Borel measures on $\To$, and recall that $\mu\in M(\To)$ is said to be a probability measure if it is nonnegative and $\int_\To\,d\mu(x)=1$.
\begin{lemma}\label{lem:Phiprob}
        Let $V\subset \mathbf{N}$ have size $n$, and suppose that $\ML(V)=\delta$. Let $\Phi$ be the function defined in \eqref{eq:Phidefi}. Then any probability measure $\mu\in M(\mathbf{T})$ satisfies $$\int_\To\Phi(x)\,d\mu(x)\geqslant 1.$$
\end{lemma}
\begin{proof}[Proof of \Cref{lem:Phiprob}]
    As $\ML(V)=\delta$ and $\phi(vx)=\mathbf{1}_{[-\delta,\delta]}(vx)=\mathbf{1}_{B(v;\delta)}(x)$ for each $v\in V$, we see that $$\Phi(x)=\sum_{v\in V}\phi(vx)=\sum_{v\in V}\mathbf{1}_{B(v;\delta)}(x)\geqslant \mathbf{1}_{\bigcup_{v\in V}B(v;\delta)}(x)=1,\quad \forall x\in\To,$$ by \eqref{eq:Bohrcover}. Hence, we find that $\int_\To\Phi(x)\,d\mu(x)\geqslant  \int_\To\,d\mu(x)=1$ as $\mu$ is a probability measure, so nonnegative with $\mu(\To)=1$.
\end{proof}
The structure vs.~randomness dichotomy that we shall employ is based on the additive dimension $\dim(V)$ of $V$. In this paper, it is convenient (though not necessary) to work with the notion of $2$-dissociativity and $\dim_2(V)$.
\begin{definition}\label{def:addidimensio}
\normalfont  
Let $k\in\mathbf{N}$. A set $D \subseteq \mathbf{Z}$ is \emph{$k$-dissociated} if whenever $$\sum_{d\in D}\varepsilon_dd=0$$
for some $\varepsilon_d\in\{-k,\dots,-1,0,1,\dots,k\}$, then all $\varepsilon_d=0$. The \emph{additive $k$-dimension} of a set $V\subset\mathbf{Z}$, denoted by $\dim_k (V)$, is the size of its largest $k$-dissociated subset $D\subset V$.\footnote{We point out that the maximal (with respect to inclusion) $k$-dissociated subsets of $V$ might not all have the same size.}
\end{definition}
The standard notions of dissociativity and additive dimension $\dim(V)$ correspond to the case where $k=1$. Throughout this paper however, we shall exclusively work with $\dim_2(V)$ and so we shall simply refer to this quantity as the additive dimension of $V$. 

\medskip

If $V$ has rather large additive dimension, then we are able to construct an explicit probability measure $\mu$ (more precisely, a Riesz product, see \eqref{eq:Rrieszdefi}) which, through applying Lemma \ref{lem:Phiprob}, witnesses that $\delta=\ML(V)$ is large. 
\begin{proposition}\label{prop:largedimintro}
    Let $V$ be a set of $n$ distinct positive integers. Then $$\ML(V)\geqslant \frac{1}{2n}+\frac{\dim_2(V)^2}{n^{3+o(1)}}.$$
\end{proposition}
If $\dim_2(V)$ is rather small, we take a different approach. Using similar ideas as in \cite{bedertsum-free}, we will show that sets $V$ with small additive dimension have a `denser model' $V'$. In the context of the Lonely Runner Conjecture, we consider a set $V'$ to be a model for $V$ if $\ML(V)= \ML(V')$, (up to a negligible error). We remark that Tao \cite[Theorem 1.3]{Tao2018LonelyRunner} proved the interesting fact that in order to confirm Conjecture \ref{conj:lonerunn1} for sets of size $n$, it suffices to check all sets $V$ of size $n$ that are contained in $[n^{O(n^2)}]$, thus in principle reducing the Lonely Runner Conjecture for a fixed number of runners to a finite computation. His argument essentially proceeds by showing that every set $V$ of size $n$ has a model $V'$ (in the sense we described above) where $V'\subset [n^{O(n^2)}]$. This bound was recently improved by Malikiosis, Santos and Schymura in \cite{malikiosis_santos_schymura2024}, showing that it suffices to check sets $V'\subset[n^{2n}]$, and it is this fact and a partial computer verification that was recently used by Rosenfeld \cite{Rosenfeld2025LonelyRunner} to verify Conjecture \ref{conj:lonerunn1} for $n=7$. 
\par The content of the next proposition is that one may find a considerably denser model if $V$ has small additive dimension. 

\begin{proposition}\label{prop:smalldimintro}
Let $V$ be a set of $n$ distinct positive integers. Then there exists a set $V'\subset \mathbf{N}$ of $n$ distinct integers satisfying:
    \begin{itemize}
        \item [(i)] $V'\subset \{1,2\dots,T\}$ where $T=e^{\dim_2(V)(\log n)^{O(1)}}$,
        \item [(ii)] $\ML(V)\geqslant \ML(V')-O(\frac{1}{n^{100}})$.
    \end{itemize}    
\end{proposition}
\begin{rmk*}
    \normalfont The choice of the error term $n^{-100}$ is rather arbitrary, and could be replaced by $n^{-C}$ for any large constant $C$. In either case, this term is negligible as all improvements for $\ML(V)$ are at least of order $n^{-2}$. 
\end{rmk*}
The final step in the argument then consist of proving a strong bound for $\ML(V)$ for sets $V$ which are contained in a relatively short initial interval $\{1,2,\dots,T\}$ of $\mathbf{N}$. Note that we only need to consider $T\leqslant e^{n^{1+o(1)}}$ by using the trivial bound $\dim_2(V)\leqslant |V|=n$ in the previous proposition. 
\begin{proposition}\label{prop:boundfordense}There exists an absolute constant $c>0$ such that the following holds. Let $V$ be a set of $n$ distinct positive integers, and suppose that $V\subset \{1,2,\dots,T\}$ for some integer $T\leqslant e^{n^{1+o(1)}}$. Then $\ML(V)\geqslant \frac{1}{2n}+\frac{c}{n\log T}$.
    
\end{proposition}
We note that some related results already exist in the literature: Tao \cite[Proposition 1.6]{Tao2018LonelyRunner} proves that if $V\subset[Cn]$ for some constant $C>0$, then one gets the linear improvement over the trivial union bound $\ML(V)\geqslant \frac{1+\Omega_C(1)}{2n}$. This is very similar to the conclusion of Proposition \ref{prop:boundfordense}, and we shall show at the end of Section \ref{sec:smalldim} that a minor modification of our argument recovers this result of Tao, in fact with a better dependence on $C$. We also note that the argument of Tao only works when $T\leqslant Cn= O(n)$, i.e.~when $V$ genuinely has positive density as a subset of $[T]$, whereas for our application it is crucial that our version permits $T$ to be as large as $e^{n^{1+o(1)}}$. Finally, we remark that Tao \cite[Proposition 1.5]{Tao2018LonelyRunner} showed that if $V\subset[ 6n/5]$ is very dense in an initial interval of $\mathbf{N}$, then $V$ satisfies the conjectured bound $\ML(V)\geqslant\frac1{n+1}$.

\medskip

Let us finish this overview by showing that Theorem \ref{th:main} can be deduced from Propositions \ref{prop:largedimintro}, \ref{prop:smalldimintro} and \ref{prop:boundfordense}. 
\begin{proof}[Proof of Theorem \ref{th:main}, assuming Propositions \ref{prop:largedimintro}, \ref{prop:smalldimintro}, and \ref{prop:boundfordense}]
    Let $V$ be an arbitrary set of $n$ distinct positive integers. If $\dim_2(V)\geqslant n^{2/3}$, then Proposition \ref{prop:largedimintro} shows that $\ML(V)\geqslant \frac{1}{2n}+\frac{1}{n^{5/3+o(1)}}$. If $\dim_2(V)\leqslant n^{2/3}$ on the other hand, then \Cref{prop:smalldimintro} allows us to find a model $V'\subset [e^{n^{2/3+o(1)}}]$ such that $\ML(V)\geqslant \ML(V')-O(1/n^{100})$. By applying Proposition \ref{prop:boundfordense} with $T=e^{n^{2/3+o(1)}}$ to $V'$, we see that $\ML(V')\geqslant \frac{1}{2n}+\frac{1}{n^{5/3+o(1)}}$. This also yields the claimed lower bound $$\ML(V)\geqslant\ML(V')-O\left(\frac1{n^{100}}\right)\geqslant \frac1{2n}+\frac{1}{n^{5/3+o(1)}}.$$ 
\end{proof}
\section{Notation and prerequisites}
We use the asymptotic notation $f=O(g)$, $f\ll g$, or $g=\Omega(f)$ if there is an absolute constant $C$ such that $|f(y)|\leqslant C g(y)$ for all $y$ in a certain domain which will be clear from context. We write $f=o(g)$ if $f(y)/g(y)\to 0$ as $y\to\infty$, and we write $f\asymp g$ if both $f\ll g$ and $g\ll f$.

\smallskip

We write $\mathbf{N},\mathbf{Z},\mathbf{R}$ and $\mathbf{C}$ for the natural, integer, real, and complex numbers, respectively. For two sets $E_1,E_2$, we define the sumset $E_1+E_2:=\{e_1+e_2:e_1\in E_1, e_2\in E_2\}$ and difference set $E_1-E_2=\{e_1-e_2:e_1\in E_1,e_2\in E_2\}$. We use the standard notation $e(x)=e^{2\pi i x}$ for $x\in \mathbf{R}$, and, as this function is $1$-periodic, it is natural to consider the domain of the variable $x$ to be $\mathbf{T}=\mathbf{R}/\mathbf{Z}$. For a suitably integrable function $g:\mathbf{T}\to\mathbf{C}$ we denote, for $p\in[1,\infty)$, its $L^p$-norm by $$\lVert g\rVert_p\vcentcolon= \left(\int_0^1|g(x)|^p\,dx\right)^{1/p},$$ and $\lVert g\rVert_\infty$ is the smallest constant $M$ such that $|g(x)|\leqslant M$ holds almost everywhere. Its Fourier coefficients are given by $\hat{g}:\mathbf{Z}\to\mathbf{C}$ which is defined by $\hat{g}(n)=\int_\mathbf{T}g(x)e(-nt)\,dx$. If $f:\mathbf{Z}\to\mathbf{C}$ is a function, we shall denote its Fourier transform by $\hat{f}(x)=\sum_{n\in\mathbf{Z}}f(n)e(nx)$ which is a priori simply a formal series.

\medskip

For two functions $g,h\in L^2(\mathbf{T})$ we define $\langle g,h\rangle =\int_\mathbf{T} g(x)\overline{h(x)}\,dx$ and we shall frequently make use of Parseval's theorem which states that $\langle g,h\rangle =\sum_{n\in\mathbf{Z}}\hat{g}(n)\overline{\hat{h}(n)}$. We also define their convolution to be the function $g*h=\To\to\mathbf{C}$ given by $(g*h)(x) = \int_\mathbf{T}g(x-y)h(y)\,dy$, and we note the basic fact that $\widehat{g*h}(n) = \hat{g}(n)\hat{h}(n)$.

\section{Lonely Runner for sets with large additive dimension}
In this section, we prove Proposition \ref{prop:largedimintro} which provides strong bounds for the gap of loneliness whenever the set $V\subset\mathbf{N}$ of speeds has rather large additive dimension. More precisely, we will show that any $V\subset \mathbf{N}$ of size $n$ satisfies 
\begin{equation}\label{eq:largedim}
\ML(V)\geqslant \frac{1}{2n}+\frac{\dim_2(V)^2}{n^{3+o(1)}}.
\end{equation}
As a warm-up, we employ this strategy to give a short proof of the following lemma, establishing the Lonely Runner Conjecture \ref{conj:lonerunn1} for dissociated sets $V$ (the  notion of $1$-dissociativity is enough here). In fact we will show not only that $\ML(V)\geqslant \frac1{n+1}$ for dissociated sets $V$, but that dissociated sets in fact have a much larger gap of loneliness $\ML(V)\gg 1/\sqrt{n}$. We remark that this lemma (or a minor variation) already recovers several results in \cite{pandey2009,czerwinski2012_random,ruzsa_tuza_voigt2002}, which study the Lonely Runner Conjecture under various lacunarity assumptions. For example, it is an elementary fact that sets $V$ satisfying $v_{j+1}\geqslant 2v_j$ for all $j$ are dissociated, and hence the following lemma immediately implies the main result of \cite{pandey2009}.
\begin{lemma}\label{lem:disslonerunn}
Let $V\subset\mathbf{N}$ be a dissociated set of size $n$. Then $\ML(V)\gg 1/{\sqrt{n}}$. 
\end{lemma}
\begin{proof}
    Consider the Riesz product $R(x)=\prod_{v\in V}(1-\cos(2\pi vx))$ which is clearly a pointwise nonnegative function on $\To$. As $\cos(2\pi vx)=\frac{e(vx)+e(-vx)}{2}$, expanding out the product shows that $$R(x)=\sum_{(\varepsilon_v)\in\{-1,0,1\}^V}\prod_{v\in V} \left(-\frac12\right)^{|\varepsilon_v|}e(\varepsilon_vvx).$$ The only contributions to the constant term come from those $V$-tuples $(\varepsilon_v)\in\{-1,0,1\}^V$ for which $\sum_{v\in V} \varepsilon_vv=0$, which is only possible if all $\varepsilon_v=0$ because $V$ is dissociated. So $\widehat{R}(0)=1$. Hence, $R(x)$ is a nonnegative function with $\int_\To R(x)\,dx=1$, so $d\mu(t):=R(t)\,dt$ is a probability measure. Let $\delta=\ML(V)$ and define $\Phi(x)$ as in \eqref{eq:Phidefi}. Lemma \ref{lem:Phiprob} then shows that $1\leqslant\int_\mathbf{T} \Phi(x)R(x)\,dx$. On the other hand, note that for any $v\in\mathbf{N}$ we can bound%\footnote{Technically, this estimate requires that $\delta\leqslant \frac12$, but we may assume this without loss of generality, as else we get the much stronger bound $\ML(V)=\frac{1}{2}$.}
    \begin{align*}\sup_{x\in\To}\phi(vx)(1-\cos(2\pi vx))&=\sup_{u\in[-\delta,\delta]}\big(1-\cos(2\pi u)\big)\\&\leqslant 1-\cos (2\pi \delta)\ll \delta^2,
    \end{align*}
    because $\phi(vx)=\mathbf{1}_{[-\delta,\delta]}(vx)$ by definition. This allows us to bound
\begin{align*}
    1&\leqslant\int_\mathbf{T} \Phi(x)R(x)\,dx\\&= \sum_{v\in V} \int_{\mathbf{T}}\bigg(\phi(vx)(1-\cos(2\pi vx))\bigg)\prod_{w\neq v}(1-\cos(2\pi wx))\,dx\\
    &\ll \delta^2\sum_{v\in V}\int_\mathbf{T}\prod_{w\neq v}(1-\cos2\pi wx)\,dx\\
    &=n\delta^2,
\end{align*}
where the final equality uses that each of the sets $V\setminus\{v\}$ is trivially still dissociated, and so has a Riesz product with constant term $1$ (using exactly the same argument as above). This yields the claimed bound that $\ML(V)=\delta\gg\frac{1}{\sqrt{n}}$.
\end{proof}
Let us now move to the general case, so consider an arbitrary fixed set $V\subset \mathbf{N}$ of size $n$ and let us write $$d=\dim_2(V)$$ throughout. By definition, we may find a subset $D\subset V$ which is $2$-dissociated and has size $|D|=d$. We will use $D$ to construct a good probability measure $\mu$ to use in Lemma \ref{lem:Phiprob}, which states that $$\int_\To\Phi(x)\,d\mu(x)\geqslant 1,$$ where $\Phi=\Phi_{V,\delta}$ is the function defined in \eqref{eq:Phidefi}. As in the proof of the previous Lemma \ref{lem:disslonerunn}, $\mu$ will be a Riesz product, chosen so that we may find a good upper bound for the integral above using Parseval's identity: $\int_\To\Phi(x)\,d\mu(x)=\sum_{m\in\mathbf{Z}}\widehat{\Phi}(m)\widehat{\mu}(m) $.\footnote{We will choose $\mu(x)=R(x)\,dx$ where $R$ is a finite Riesz product (and hence a trigonometric polynomial), so there are no convergence issues.} As we eventually want to bound such a sum, we need to study the Fourier series of $\Phi$, so let us begin with this. Recall as in \eqref{eq:Phidefi} that $$\Phi(x)=\sum_{v\in V}\phi(vx)$$ where $\phi(x)=\mathbf{1}_{[-\delta,\delta]}(x)$, and $\delta=\ML(V)$. The function $\phi:\mathbf{T}\to \mathbf{R}$ is a simple step function and has the following Fourier expansion:
\begin{equation*}
    \phi(x)\sim2\delta+\sum_{m=1}^\infty\frac{\sin(2\pi m\delta)}{\pi m}\left(e( mx)+e(-mx)\right).
\end{equation*}
Basic Fourier analysis confirms that both sides are equal in $L^2(\To)$, and in fact that equality holds for all $x\in\To$ except $x=\pm \delta$ (the discontinuity points). Let us write $\lambda(m)=\frac{\sin (2\pi m\delta)}{\pi m}$ so that
\begin{equation*}
    \phi(x)=2\delta+\sum_{m=1}^\infty\lambda(m)(e( mx)+e(-mx)).
\end{equation*}
Note that we can therefore write \begin{equation}\label{eq:Fexpansion}
\Phi(x)=\Phi_{V,\delta}(x)=2\delta n+\sum_{v\in V}\sum_{m=1}^\infty \lambda(m)(e(mvx)+e(-mvx)),
\end{equation}
where we interpret this equality in $L^2$ (as this will be enough for our purposes). Our first goal is to show that $\Phi$ has many Fourier coefficients that are quite large, which due to the scaling\footnote{Note that $\lambda(1)\asymp \delta$, and that the interesting sets in the context of Conjecture \ref{conj:lonerunn1} are those $V$ for which $\frac{1}{2n}\leqslant \ML(V)=\delta\leqslant \frac{1}{n+1}$. } means that we hope to show that many Fourier coefficients of $\Phi$ have size $\gg 1/n$. This is precisely what the following lemma achieves.
\begin{lemma}\label{lem:largeFouriercoeff}
Let $V$ be a set of $n$ distinct positive integers, and define $\Phi$ as in \eqref{eq:Phidefi}. If $\ML(V)\leqslant \frac{0.9}{n}$, then for each $v\in V$ we have that \begin{equation}\label{eq:largeFouriercoeff}
\sum_{j=1}^{100n}\left(1-\frac{j}{n}\right)\widehat{\Phi}(jv)\geqslant \frac{1}{100}.
\end{equation}
\end{lemma}
\begin{rmk*}
    \normalfont Note that if $\ML(V)>\frac{0.9}{n}$, then we immediately have a much stronger bound than what we claim in Proposition \ref{prop:largedimintro} (as this would even be a constant factor improvement over the trivial bound $\frac1{2n}$), so we have decided to include the extra assumption that $\ML(V)\leqslant \frac{0.9}{n}$ in the lemma since it allows for a cleaner proof. 
\end{rmk*}
So while we cannot say that $\widehat{\Phi}(v)\gg \lambda(1)\gg 1/n$ for each $v\in V$ (which would be the case if the Fourier series of all of the individual functions $\phi(vx)$ appearing in \eqref{eq:Fexpansion} did not interfere too much with each other), this lemma confirms that many of the Fourier coefficients $\widehat{\Phi}(jv)$ for $j\in [100n]$ are $\gg 1/n$, and this good enough for our purposes.
\begin{proof}[Proof of Lemma \ref{lem:largeFouriercoeff}]
Define $K_m(x)\vcentcolon= \sum_{j=-m}^m(1-|j|/m)e(jx)$ to be the Fej\'er kernel of degree $m-1$. We claim that it is enough to prove for each fixed $v\in V$ that \begin{align}\label{eq:FejerF}\int_\mathbf{T} (\Phi(x)-2\delta n)K_{100n}(vx)\,dx\geqslant \frac{1}{50}.
\end{align}
Indeed, if this holds, then by Parseval and as it is clear from \eqref{eq:Fexpansion} that the constant term in $\Phi(x)-2\delta n$ vanishes, we get that $$\sum_{\substack{j\in[-100n,100n]\\j\neq 0}}^{}\left(1-\frac{|j|}{100n}\right)\widehat{\Phi}(jv)=\int_\mathbf{T}(\Phi(x)-2\delta n)K_{100n}(vx)\,dx\geqslant \frac{1}{50}.$$ As $\widehat{\Phi}(-k)=\widehat{\Phi}(k)$ for all $k$ (since $\Phi$ is even), this implies \eqref{eq:largeFouriercoeff}.

\medskip

It remains to prove \eqref{eq:FejerF} for each fixed $v\in V$. Two basic properties of the Fej\'er kernel $K_m(x)$ that we will use are that $K_m$ is a pointwise nonnegative function, for any $m$, and that $\int_\mathbf{T}K_m(x)\,dx=\widehat{K}_m(0)=1$. Definition \eqref{eq:Phidefi} states that $\Phi(x)-2\delta n=-2\delta n+\sum_{w\in V}\phi(wx)$, so we can calculate \begin{align}
    \int_\mathbf{T}(\Phi(x)-2\delta n)K_{100n}(vx)\,dx=&-2\delta n\int_\mathbf{T}K_{100n}(vx)\,dx\nonumber\\
    &+\int_\mathbf{T}\phi(vx)K_{100n}(vx)\,dx\label{eq:3terms}\\
    &+\sum_{w\in V:w\neq v}\int_\mathbf{T}\phi(wx)K_{100n}(vx)\,dx.\nonumber
\end{align}
    The first term in \eqref{eq:3terms} is $(-2\delta n)\widehat{K}_{100n}(0)=-2\delta n$. By substituting $vx=u$ and recalling that $\phi=\mathbf{1}_{[-\delta,\delta]}$, we see that the second term in \eqref{eq:3terms} equals
    \begin{align*}
        \int_\mathbf{T}\phi(u)K_{100n}(u)\,du&=\int_{-\delta}^{\delta}K_{100n}(u)\,du\\
        &=\int_\mathbf{T}K_{100n}(u)\,du-2\int_{\delta}^{1/2} K_{100n}(u)\,du\\
        &= 1-2\int_{\delta}^{1/2}\frac{\sin^2(100n\pi u)}{100n\sin^2(\pi u) }\,du\\
        &\geqslant 1-\frac{2}{100n}\int_{\delta}^{1/2}\frac{du}{4u^2}\\
        &\geqslant 1-\frac{1}{200n}\times \frac{1}{\delta},
    \end{align*}
     using that the Fej\'er kernel has closed form $K_m(x)=\sin^2(\pi m x)/(m\sin^2(\pi x))$, and that $\sin(\pi x)\geqslant 2x $ for $x\in [0,1/2]$. The trivial union bound states that $\delta=\ML(V)\geqslant \frac{1}{2n}$, so using this in the previous bound shows that $\int_\mathbf{T}\phi(vx)K_{100n}(vx)\,du\geqslant 0.99$. Finally, we claim that the third term in \eqref{eq:3terms} is at least $\delta (n-1)$, which then gives in total that $$\int_\mathbf{T}(\Phi(x)-2\delta n)K_{100n}(x)\,dx\geqslant -2\delta n+0.99+\delta (n-1)=0.99-\delta (n+1).$$ As we are assuming that $\delta=\ML(V)\leqslant 0.9/n$, this integral is at least $1/50$, thus proving \eqref{eq:FejerF}. 
     
     \medskip
     
     To prove the final claim that the third term in \eqref{eq:3terms} can be bounded from below by $$\sum_{w\in V:w\neq v}\int_\mathbf{T}\phi(wx)K_{100n}(vx)\,dx\geqslant \delta (n-1),$$ it suffices to prove the general fact that $\int_{\mathbf{T}}\mathbf{1}_I(ax)K_m(bx)\,dx\geqslant |I|/2$ for all integers $a,b,m\in \mathbf{N}$ and all intervals $I=[-\eta,\eta]\subset\To$. Indeed, we recall that $\phi=\mathbf{1}_{[-\delta,\delta]}$ so this fact implies that $\int_\mathbf{T}\phi(wx)K_{100n}(vx)\,dx\geqslant \delta $ for all $w\in V$. Now to see why this fact holds, note that if we define $\chi(x)=(2|I|^{-1})(\mathbf{1}_{I/2}*\mathbf{1}_{I/2})(x)$ where $I/2:=[-\frac{\eta}{2},\frac{\eta}{2}]$, then $\chi$ is a tent function on $\To$ which is pointwise bounded from above by $\mathbf{1}_I$. Hence, as $K_m$ is pointwise nonnegative, we get the lower bound
     \begin{align*}
         \int_\mathbf{T}\mathbf{1}_I(ax)K_m(bx)\,dx&\geqslant 2|I|^{-1}\int_\mathbf{T}(\mathbf{1}_{I/2}*\mathbf{1}_{I/2})(ax)K_m(bx)\,dx\\
         &=2|I|^{-1}\sum_{k\in \mathbf{Z}:a,b\mid k}|\widehat{\mathbf{1}}_{I/2}(k/a)|^2\widehat{K}_m(k/b)\\
         &\geqslant 2|I|^{-1}|\widehat{\mathbf{1}}_{I/2}(0)|^2\widehat{K}_m(0)=|I|/2, 
     \end{align*} by Parseval, and because all Fourier coefficients of $\mathbf{1}_{I/2}*\mathbf{1}_{I/2}$ and $K_m$ are nonnegative so that we get a lower bound by ignoring all terms except the one with frequency $0$.
\end{proof}
Our goal in this section is to prove Proposition \ref{prop:largedimintro}, giving the bound \eqref{eq:largedim}. This is trivially true if $\ML(V)\geqslant \frac{0.9}{n}$, so by Lemma \ref{lem:largeFouriercoeff} we may without loss of generality assume throughout the rest of this section that $$\sum_{j=1}^{100n}\left(1-\frac{j}{n}\right)\widehat{\Phi}(jv)\geqslant \frac1{100},$$ for each $v\in V$. Hence, recalling that the $2$-dissociated set $D$ is a subset of $V$ of size $d=\dim_2(V)$, we obtain 
$$\sum_{m\in D}\sum_{j=1}^{100n}\left(1-\frac{j}{n}\right)\widehat{\Phi}(jm)\geqslant \frac{d}{100}.$$
After interchanging the sums, it is clear that there exists some $j\in[100n]$ such that
\begin{equation}\label{eq:DlargeFourier}
    \sum_{m\in D}\widehat{\Phi}(jm)\geqslant \Omega\left(\frac{d}{n}\right).
\end{equation}
We fix a single such $j$ for the rest of our argument, and let us define
\begin{equation}\label{eq:D'defi}
    D':=j\cdot D=\{jm:m\in D\}.
\end{equation}
It is trivial that $D'$ is itself $2$-dissociated (being a dilate of a $2$-dissociated set) and of size $|D'|=|D|=d$.
We now consider the Riesz product measure $d\mu(x) = R(x)\,dx$ with 
\begin{equation}\label{eq:Rrieszdefi}
    R(x)=\prod_{m\in D'}\big(1-p\cos(2\pi mx)\big),
\end{equation}
where $p\in[0,1]$ is some constant to be chosen later.\footnote{It might be helpful to keep in mind that we will choose $p\asymp d/n^{1+o(1)}$, where $d=|D|=\dim_2(V)$.} Note that
\begin{itemize}
    \item [(i)] $R(x)$ is pointwise nonnegative for $x\in \To$,
    \item [(ii)] and $\int_\To R(x)\,dx=\widehat{R}(0)=1$.
\end{itemize} Property (i) is trivial as $R(x)$ is a product of terms of the form $1-p\cos(2\pi mx)$, which are nonnegative as $p\in[0,1]$. In the proof of Lemma \ref{lem:disslonerunn}, we already saw that the constant term of a Riesz product $R$ of a dissociated set is $1$, which implies (ii). We prove this again and moreover find a general expression for the Fourier coefficients of $R$. Define the sets
\begin{equation}\label{eq:E_kdefi}
    E_k=\left\{\sum_{m\in D'}\varepsilon_mm:\varepsilon\in\{-1,0,1\} \text{ and }\sum_m|\varepsilon_m|=k\right\}.
\end{equation} The fact that $D'$ is $2$-dissociated shows that the only solutions $(\varepsilon_m),(\eta_m)\in\{-1,0,1\}^{D'}$ to $\sum_{m\in D'}\varepsilon_mm=\sum_{m\in D'}\eta_mm$ are the trivial solutions where $(\varepsilon_m)=(\eta_m)$. Hence, we see that $E_k$ consists of precisely $\binom{n}{k}2^k$ distinct integers, and that the sets $E_k$ for $0\leqslant k\leqslant d=|D'|$ are pairwise disjoint. As $\cos(2\pi mx)=\frac{e(mx)+e(-mx)}{2}$, simply expanding out the product \eqref{eq:Rrieszdefi} shows that
\begin{align*}R(x)&=\sum_{(\varepsilon_m)\in\{-1,0,1\}^{D'}}\prod_{m\in D'} \left(\frac{-p}{2}\right)^{|\varepsilon_m|}e(\varepsilon_mmx)\\
&=\sum_{k=0}^n\left(\frac{-p}{2}\right)^k\sum_{\ell\in E_k}e(\ell x).
\end{align*}
%so that 
%$$\widehat{R}(m)=\sum_{(\varepsilon_z)\in\{-1,0,1\}^{D'}:\sum_z\varepsilon_zz=m}\left(\frac{-p}{2}\right)^{\sum_{z}|\varepsilon_z|},$$ 
Note that $E_0=\{0\}$ and that $0\notin E_k$ for $k\in\{1,\dots,d\}$, as $D'$ is $2$-dissociated, which again confirms property (ii) that $\int_\To R(x)\,dx=\widehat{R}(0)=1$. In fact, we have shown that $\widehat{R}(\ell)=(-p/2)^k$ if and only if $\ell\in E_k$. Properties (i) and (ii) confirm that $\mu(x)=R(x)\,dx$ is a probability measure, so Lemma \ref{lem:Phiprob} shows that
\begin{equation}\label{eq:crucialineq}
    1\leqslant \int_\To \Phi(x)R(x)\,dx.
\end{equation}
The remainder of the proof of Proposition \ref{prop:largedimintro} consists of using the properties of $R$ to find a good upper bound for the integral in  \eqref{eq:crucialineq}. We begin by applying Parseval:
\begin{align}
1\leqslant \int_\To\Phi(x)R(x)\,dx&=\sum_{m\in\mathbf{Z}}\widehat{\Phi}(m)\widehat{R}(m)\nonumber\\
&=2\delta n-\underbrace{\frac{p}{2}\sum_{m\in E_1}\widehat{\Phi}(m)}_{T_2}+\underbrace{\sum_{k=2}^d\left(\frac{-p}{2}\right)^k\sum_{m\in E_k}\widehat{\Phi}(m)}_{T_3},\label{eq:Riesztest}
\end{align}
where we used that $\widehat{\Phi}(0)=2\delta n$ and $\widehat{R}(0)=1$ to determine the contribution of $m=0$, and that $\widehat{R}(m)=(-p/2)^k$ if $m\in E_k$ while $\widehat{R}(m)=0$ if $m\notin \bigcup_{k=0}^dE_k$. As we defined $D'=j\cdot D$ in \eqref{eq:D'defi}, and as $E_1=D'\cup-D'$, we can estimate the contribution of the second term in \eqref{eq:Riesztest} as follows:
\begin{align}\label{eq:T_2defi}
    T_2=-\frac{p}{2}\sum_{m\in E_1}\widehat{\Phi}(m)=-\frac{p}{2}\sum_{m\in D}(\widehat{\Phi}(jm)+\widehat{\Phi}(-jm))\leqslant -\Omega\left(\frac{pd}{n}\right),
\end{align} by using \eqref{eq:DlargeFourier} (and that $\widehat{\Phi}(m)=\widehat{\Phi}(-m)$ as $\Phi$ is an even function). We will show that the contribution of the third term $T_3$ in \eqref{eq:Riesztest}, coming from those $E_k$ with $k\geqslant 2$, is rather small. For this, we use the following result. 
\begin{lemma}\label{lem:PE_kbound}
    Let $A\subset\mathbf{Z}$ be a $2$-dissociated set, and define \begin{equation}\label{eq:AE_kdefi}
        E_k=\left\{\sum_{a\in A}\varepsilon_aa:\varepsilon\in\{-1,0,1\} \text{ and }\sum_a|\varepsilon_a|=k\right\}.
    \end{equation} Then any arithmetic progression $P$ contains at most $|E_k\cap P|\leqslant (C\log |P|)^{k}$ elements of $E_k$,
    where $C$ is an absolute constant.
\end{lemma}
\begin{rmk*}
    \normalfont This is best possible (up to the value of $C$), as can be seen by considering the $2$-dissociated set $A=\{1,3,\dots,3^{\lfloor\log_3N\rfloor}\}\subset [N]$.
\end{rmk*}
It is known that results like Lemma \ref{lem:PE_kbound} follow from a version of Bonami's influential inequality \cite[Theorem 5, p.359]{bonami} (which states that the set $E_k$ in \eqref{eq:AE_kdefi} is a so-called $\Lambda(q)$-set for each $q>2$), and a classical argument of Rudin \cite[Theorem 3.5]{rudintrig} showing that $\Lambda(q)$-sets cannot contain too many elements of an arithmetic progression. This is perhaps easier to see by looking at Theorem 5.13 (Bonami's inequality) and Theorem 6.3 (Rudin's result) in the book \cite{lopezross} of Lopez and Ross. As stated in these sources, these results imply that $|E_k\cap P|<C(k)(\log |P|)^k$ without making the dependence of $C(k)$ on $k$ explicit, though their arguments show that one may take $C(k)=C^k$ for some absolute constant $C$. For our application, it also suffices\footnote{The bound $\frac{1}{2n}+\frac{\dim_2(V)}{n^{3+o(1)}}$ in Proposition \ref{prop:largedimintro} is in fact of the form $\frac{1}{2n}+\frac{\dim_2(V)}{n^{3}(\log n)^{C''}}$, and using this slightly weaker version of the Bonami-Rudin bound only results in a minor increase in the value of $C''$.} to have the slightly weaker bound \begin{align}\label{eq:E_kP2}
    |E_k\cap P|\leqslant (C_1(\log |A|)(\log |P|))^k,
\end{align}
for some absolute constant $C_1$ (independent of $A,P$ and $k$). For completeness, we include a short combinatorial proof of \eqref{eq:E_kP2} in the \hyperref[sec:appendix]{Appendix}, based on some progress on the sunflower problem of Erd\H{o}s and Rado. We will use Lemma \ref{lem:PE_kbound} (or \eqref{eq:E_kP2}) to find a good bound for the third term in \eqref{eq:Riesztest}, and we begin by applying the triangle inequality:
\begin{align}\label{eq:T_3defi}
    T_3=\sum_{k=2}^d\left(\frac{-p}{2}\right)^k\sum_{m\in E_k}\widehat{\Phi}(m)&\leqslant \sum_{k=2}^dp^k\sum_{m\in E_k}|\widehat{\Phi}(m)|.
\end{align}
Note by \eqref{eq:Fexpansion} that $$\widehat{\Phi}(m)=\sum_{v\in V}\lambda\left(\frac{m}{v}\right)\mathbf{1}_{\{v\text{ divides } m\}},$$ where we recall that $\lambda(\ell)=\frac{\sin(2\pi \ell \delta)}{\ell\pi}$. So $|\widehat{\Phi}(m)|\leqslant \sum_{v\in V: v\text{ divides } m}|\lambda\left(\frac{m}{v}\right)|$, and by using this in \eqref{eq:T_3defi} we obtain the bound
\begin{align}\label{eq:T_3vbound}
    T_3\leqslant \sum_{v\in V}\sum_{k=2}^dp^k\sum_{j=-\infty}^\infty |\lambda(j)|\mathbf{1}_{E_k}(jv)\leqslant2n\max_{v\in V}\left(\sum_{k=2}^dp^k\sum_{j=1}^\infty |\lambda(j)|\mathbf{1}_{E_k}(jv)\right)\,
\end{align} where we added an extra factor of 2 in front to restrict the inner sum to positive $j$, which is permissible as $0\notin \bigcup_{k=1}^dE_k$ and as the sets $E_k=-E_k$ are symmetric, while $|\lambda(m)|=|\lambda(-m)|$ for all $m\in\mathbf{Z}$. Let $w\in V$ attain the maximum in \eqref{eq:T_3vbound}. The trivial union bound shows that $\delta\geqslant 1/2n$, and we may also assume without loss of generality that $\delta\leqslant \frac{1}{n+1}$ (else $V$ satisfies the Lonely Runner Conjecture, so certainly the conclusion of Proposition \ref{prop:largedimintro}). Hence, $\delta\asymp 1/n$. As $|\sin(2\pi j\delta)|\leqslant \min(2\pi j\delta, 1)\ll \min(j/n,1)$, it follows that
    \begin{equation}\label{eq:lambdabound}
        |\lambda(j)|=\frac{|\sin(2\pi j\delta)|}{j\pi}\ll \begin{cases} \frac{1}{n} &\text{ if }j\in [n],\\
        \frac{1}{j} &\text{ for all }j.
            
        \end{cases}
    \end{equation}
We will split the sum over $k$ in \eqref{eq:T_3vbound} into the two ranges $[2,10\log n)$ and $[10\log n,d]$, and estimate these separately. First, in the latter range the factor $p^k$ is already tiny so that it suffices to use the simple bound
\begin{align}
        \sum_{k=10\log n}^dp^k\sum_{j=1}^\infty |\lambda(j)|\mathbf{1}_{E_k}(jw)&\ll p^{10\log n}\sum_{j=1}^\infty\frac{1}{j }\times \mathbf{1}_{\{jw\in \bigcup_{k\geqslant 2}E_k\}}\nonumber\\
        &\ll p^{10\log n}\times \log\left|\bigcup_{k\geqslant 2}E_k\right|\nonumber\\
        &\ll np^{10 \log n},\label{eq:largekbound}
\end{align} where the first inequality uses that the sets $E_k$ are pairwise disjoint (as $D'$ is $2$-dissociated) so that each term $|\lambda(j)\ll1/j$ appears at most once, and for the final bound we note that $\left|\bigcup_{k=0}^dE_k\right|=3^d\leqslant 3^n$ which is clear from the definition \eqref{eq:E_kdefi}. We write $w\cdot [i_1,i_2]$ for the arithmetic progression $$w\cdot [i_1,i_2]:=\{jw:j\in\{i_1,i_1+1,\dots,i_2\}\},$$ and recall that each such progression $w\cdot [i_1,i_2]$ contains at most $(C (\log n)(\log (i_2-i_1))^{k}$ elements of $E_k$ by \eqref{eq:E_kP2}, while Lemma \ref{lem:PE_kbound} gives an even stronger bound. Using this and the bound \eqref{eq:lambdabound} for the coefficients $\lambda(j)$ allows us to bound the contribution from $k\in[2,10\log n)$ as follows:
\begin{align*}
        \sum_{k=2}^{10\log n}p^k\sum_{j=1}^\infty |\lambda(j)|\mathbf{1}_{E_k}(jw)&\ll\sum_{k=2}^{10\log n}p^k\left(\frac{\big| E_k\cap (w\cdot[n])\big|}{n}+\sum_{\ell=1}^\infty \frac{1}{n^\ell }\big|E_k\cap (w\cdot[n^{\ell},n^{\ell+1} ))\big|\right)\\
        &\ll \frac{1}{n}\sum_{k=2}^{10\log n}p^{k}(C\log n)^{2k}\left(1+\sum_{\ell=1}^\infty \frac{(\ell+1) ^{k}}{n^{\ell-1} }\right)\\
        &\ll \frac{1}{n}\sum_{k=2}^{10\log n}p^{k}(C\log n)^{2k}\left(1+\sum_{\ell=1}^\infty \frac{\ell ^{k}}{e^{\ell} }\right)\\
        &\ll \frac{1}{n}\sum_{k=2}^{10\log n}p^{k}(C\log n)^{2k}\left(1+k^{k}\right)\ll\frac{1}{n} \sum_{k=2}^\infty(Cp^{1/3}\log n)^{3k},
    \end{align*} where the precise value of the absolute constant $C$ may vary from line to line, and where the penultimate inequality uses that $\sum_{\ell=1}^\infty \ell^{k}e^{-\ell}\ll \int_0^\infty x^{k}e^{-x}dx=\Gamma(k+1)$. Using this bound and the bound \eqref{eq:largekbound} in \eqref{eq:T_3vbound} shows that $$T_3\ll n^2p^{10\log n}+\sum_{k=2}^\infty(Cp^{1/3}\log n)^{3k}.$$ Plugging this and our earlier estimate for $T_2$ from \eqref{eq:T_2defi} into \eqref{eq:Riesztest} yields
    \begin{align}\label{eq:finaldelta}
       1\leqslant 2\delta n+T_2+T_3\leqslant 2\delta n-\Omega\left(\frac{pd}{n}\right)+n^2p^{10\log n}+\sum_{k=2}^\infty(Cp^{1/3}\log n)^{3k}.
    \end{align}
    Finally, we choose the parameter $p:= \frac{d}{n(C'\log n)^{7}}$, where $C'$ is a sufficiently large absolute constant. So certainly $p<1/e$ and the term $n^2p^{10\log n}$ is at most $n^{-8}$, which is (basically) negligible. The term $\sum_{k=2}^\infty(Cp^{1/3}\log n)^{3k}$ becomes a geometric series with ratio less than $1/2$ and so is bounded (up to a constant factor) by its first term, which is $O(p^2\log^6n)$. Using this and rearranging \eqref{eq:finaldelta} gives
    \begin{align*}
        \delta&\geqslant \frac{1}{2n}+\Omega\left(\frac{pd}{n^2}\right)-O\left(n^{-9}+\frac{p^2(\log n)^6}{n}\right)\\
        &\geqslant \frac{1}{2n}+\Omega\left(\frac{d^2}{n^3(\log n)^{7}}\right) -O\left(\frac{d^2}{n^3(\log n)^{8}}\right)\\
        &\geqslant \frac{1}{2n}+\frac{d^2}{n^{3+o(1)}},
    \end{align*}
    completing the proof of Proposition \ref{prop:largedimintro}. One can remove some of the log-factors that make up the $n^{o(1)}$-term if one is a bit more careful.

\section{Lonely Runner for sets with small additive dimension}\label{sec:smalldim}
In this section, we prove 
\begin{itemize}
    \item Proposition \ref{prop:smalldimintro}, which shows that sets $V$ with rather small dimension may be replaced by a `denser model' $V'$, by which we mean that $V'\subset\{1,2\dots, T\}$ for some relatively small $T$, for the purpose of estimating $\ML(V)$, 
    \item Proposition \ref{prop:boundfordense}, which proves good bounds for $\ML(V')$ for sets $V'$ which are contained in a rather short initial segment of $\mathbf{N}$.
\end{itemize}
Together, these propositions provide good bounds for $\ML(V)$ whenever $V$ has rather small dimension. 

\medskip

To prove Proposition \ref{prop:smalldimintro}, we follow the strategy from section 6 in the author's work \cite{bedertsum-free} on sum-free sets, which already contains most of the main ideas for obtaining such a `denser model' for sets with small dimension. We need to adapt some of the combinatorial lemmas from \cite{bedertsum-free} to the setting of the Lonely Runner Conjecture. 

\medskip

Here, we will write $\dim_2^-(B)$ for the size of the smallest (in size) maximal (with respect to inclusion) $2$-dissociated subset of $B$. It is trivial that $\dim_2^-(B)\leqslant \dim_2(B)$. We also define the 2-span of a set $B$ to be $\Span_2(B)=\{\sum_{b\in B}\varepsilon_b b:\varepsilon_b\in \{-2,-1,0,1,2\}\}$, and write $\lambda\cdot B:=\{\lambda b:b\in B\}$ for the dilate of a set.
\begin{lemma}\label{lem:small-dim-rectification}
If $B \subseteq \mathbf{Z}/p\mathbf{Z}$ and $\dim_2^-(B)\leqslant d$, then there is a $\lambda \in (\mathbf{Z}/p\mathbf{Z})^\times$ such that the dilate $\lambda\cdot B$ is contained in the interval $[-8dp^{1-1/(2d)},8dp^{1-1/(2d)}]$.
\end{lemma}
\begin{proof}
Let $D_0$ be a $2$-dissociated subset of $B$ of size $|D_0|=\dim_2^-(B)\leqslant d$ which is maximal with respect to inclusion. Define $D:=D_0\cup \{z/2:z\in D_0\}\subset \mathbf{Z}/p\mathbf{Z}$, and it is clear that $|D|\leqslant 2d$. We claim that $B\subset \Span_2(D)$. To see this, note that if $b\in B\setminus D_0$, then $D_0\cup\{b\}$ is a strict superset of $D_0$, so by the maximality of $D_0$ there must exist a relation $$\varepsilon_bb+\sum_{z\in D_0}\varepsilon_zz=0\md{p},$$ for some $\varepsilon_b,\varepsilon_z\in\{-2,-1,0,1,2\}$. As $D_0$ is $2$-dissociated, it must be the case that $\varepsilon_b\neq 0$, which shows that $b\in\Span_2(D)$. 

\medskip

Now consider the set
$$\{(\lambda z/p)_{z\in D}: \lambda \in \mathbf{Z}/p\mathbf{Z}\} \subseteq \mathbf{T}^{2d}.$$
We divide $\mathbf{T}^{2d}$ into boxes of the form $\prod_{i=1}^{2d}[j_i/m,(j_i+1)/m)$ where $m:=\lceil p^{1/(2d)}/2\rceil$ and the $j_i$ range over the integers in $[0,m)$. As there are $m^{2d}<p$ boxes in total, the pigeonhole principle provides distinct $\lambda_1, \lambda_2 \in \mathbf{Z}/p\mathbf{Z}$ such that the vectors $(\lambda_1 z/p)_z,(\lambda_2 zd
/p)_z$ lie in the same box, so $\|\lambda_1 z/p-\lambda_2 z/p\|_{\mathbf{T}} \leqslant 1/m\leqslant2p^{-1/(2d)}$ for all $z\in D$.  Take $\lambda\vcentcolon=\lambda_1-\lambda_2 \in (\mathbf{Z}/p\mathbf{Z})^\times$, so that $\lambda z \in [-2p^{1-1/(2d)},2p^{1-1/(2d)}]$ for all $z\in D$.  Since $B\subseteq \operatorname{span}_2(D)$, we deduce that
$\lambda \cdot B \subseteq [-8dp^{1-1/(2d)},8dp^{1-1/(2d)}].$
\end{proof}
The next lemma will allow us to find, given a set $V$, another set $V'$ for which $\ML(V)\approx\ML(V')$ (up to a negligible error), and where $V'$ is contained in a shorter interval of $\mathbf{Z}$. For a prime $p$ (which will be clear from context), we write $\pi:\mathbf{Z}\to\mathbf{Z}/p\mathbf{Z}:m\mapsto m\md{p}$, and we shall always use the convention that $\pi^{-1}$ takes values in $(-p/2,p/2)$.
\begin{lemma}\label{lem:dilating}
    Let $B\subset\{1,2,\dots,p-1\}\subset\mathbf{Z}$, where $p$ is a prime. Suppose that $\lambda\in(\mathbf{Z}/p\mathbf{Z})^\times$ satisfies $\lambda\cdot \pi(B)\subset(-\frac{p}{\ell},\frac{p}{\ell})\subset\mathbf{Z}/p\mathbf{Z}$, for some $\ell\geqslant 2$. Write $B':=\pi^{-1}(\lambda\cdot\pi(B))$. Then $B'\subset (-\frac{p}{\ell},\frac{p}{\ell})\subset\mathbf{Z}$ and $$\ML(B)\geqslant \ML(B')-\frac{1}{\ell}.$$
    Moreover, if $\ell>2|B|$, then $\dim_2^-(B')\leqslant \dim_2^-(B)$.
\end{lemma}
\begin{proof}
    As $\lambda\cdot \pi(B))\subset \left(-\frac{p}{\ell},\frac{p}{\ell}\right)\subset\mathbf{Z}/p\mathbf{Z}$ by assumption, it is immediate that $B'\subset \pi^{-1}((\frac{-p}{\ell},\frac{p}{\ell}))=(\frac{-p}{\ell},\frac{p}{\ell})\subset\mathbf{Z}$. Let $t'\in\mathbf{T}$ be such that $\min_{b'\in B'}\lVert t'b'\rVert_\mathbf{T}=\ML(B')$. We can find an integer $a$ such that $|t'-a/p|<1/p$, and we claim that $t=\lambda a/p$ witnesses that $\ML(B)\geqslant \ML(B')-1/\ell$. Indeed, pick any $b\in B$, then \begin{align*}
    \lVert tb\rVert_\mathbf{T}&= \lVert a\lambda b/p\rVert_\mathbf{T}\\
    &=\lVert ab'/p\rVert_\mathbf{T}
    \end{align*} for some $b'\in B'$, since $\pi(B')=\lambda\cdot \pi(B)$ modulo $p$. Hence,
     \begin{align*}
        \lVert tb\rVert_\mathbf{T}=\lVert ab'/p\rVert_\mathbf{T}&\geqslant \lVert t'b'\rVert_\mathbf{T}-|t'-a/p|\times\lVert b'\rVert_\mathbf{T}\\
        &\geqslant \ML(B')-1/\ell,
    \end{align*} because $B'\subset (-\frac{p}{\ell},\frac{p}{\ell})$ and because $|t'-a/p|<1/p$.

    \medskip

    Let us now prove that $\dim_2^-(B')\leqslant \dim_2^-(B)$, under the assumption that $\ell>2|B|$. It is trivial that $\dim_2^-(\pi(B))\leqslant \dim_2^-(B)$ (the first notion of dimension is considered over $\mathbf{Z}/p\mathbf{Z}$, the second over $\mathbf{Z}$), and also that $\dim_2^-$ is invariant under dilations, meaning that $\dim_2^-(\pi(B))=\dim_2^-(\mu\cdot \pi(B))$ for each $\mu\in(\mathbf{Z}/p\mathbf{Z})^\times$. Let us write $k:=\dim_2^-(\pi(B))$, so that $\dim_2^-(\lambda\cdot \pi(B))=k\leqslant \dim_2^-(B)$. By the definition of $\dim_2^-$, there exists a $2$-dissociated subset $C\subset \lambda\cdot \pi(B)$ of size $k$ which is maximal with respect to inclusion. This means that for every $b'\in B'=\pi^{-1}(\lambda\cdot \pi(B))$, there exist coefficients $\varepsilon_{b'},\varepsilon_c\in\{-2,-1,0,1,2\}$, not all $0$, such that $$\varepsilon_{b'}b'+\sum_{c\in \pi^{-1}(C)}\varepsilon_cc\equiv 0\md{p}.$$ As $B'\subset (\frac{-p}{\ell},\frac{p}{\ell})\subset(\frac{-p}{2|B|},\frac{p}{2|B|})$ because of our assumption that $\ell>2|B|$, and as $\pi^{-1}(C)\subset B'$ since $C\subset \lambda\cdot \pi(B)$, the left hand side of this congruence is some integer in $(-p,p)$, implying that $$\varepsilon_{b'}b'+\sum_{c\in \pi^{-1}(C)}\varepsilon_cc= 0.$$ This shows that $\pi^{-1}(C)$ is a maximal $2$-dissociated subset of $B'$ (over $\mathbf{Z}$) with respect to inclusion, so that $\dim_2^-(B')\leqslant k$ as desired.
\end{proof}
An iterative application of these lemmas allows us to prove Proposition \ref{prop:smalldimintro}, though some care is needed to control the various parameters. 
\begin{proof}[Proof of Proposition \ref{prop:smalldimintro}]
    Let $V$ be a set of $n$ distinct positive integers, with $\dim_2(V)=d$. Our goal is to show that there exists a `model' $V'\subset \{1,2,\dots,T\}$ where $T\leqslant e^{d(\log n)^{O(1)}}$ such that \begin{equation*}
        \ML(V)\geqslant \ML(V')-O\left(\frac{1}{n^{100}}\right).
    \end{equation*}
    We may suppose that $\max V>e^{d(\log n)^{4}}$, as else we're done by simply taking $V'=V$ as the desired model. Now define the positive integer $m$ to be the minimal integer which satisfies the following three properties:
    \begin{itemize}
        \item [(i)] $m>e^{d(\log n)^{4}}$,
        \item [(ii)] there exists a set $V'\subset\mathbf{N}$ with $\max(V')\leqslant m$, i.e.~$V'\subset [m]$, such that 
    \begin{equation}\label{eq:MLclose}
    \ML(V)\geqslant \ML(V')-\frac{1}{m^{\varepsilon}},
    \end{equation}
    where we define $\varepsilon=\varepsilon(d,n):=\frac{1}{d(\log n)^{2}}$,
        \item [(iii)] and the set $V'$ also satisfies $\dim_2^-(V')\leqslant d$.
    \end{itemize} 
    Note that $\varepsilon$ is simply a constant as we consider $V$ (and hence $n=|V|$ and $d=\dim_2(V)$) to be fixed throughout this proof. Further note that this minimum $m$ is well-defined as taking $V'=V$ and $m=\max(V)$ gives one valid choice satisfying (i),(ii) and (iii). Let us now consider this set $V'$ for which $\max (V')=m$ is minimal among all choices satisfying (i),(ii) and (iii). Choose a prime $p\in(2m,4m]$ and define $$V'_p\vcentcolon= \pi(V')\subset\mathbf{Z}/p\mathbf{Z}.$$ Note that since $V'\subseteq[m]\subset[p/2]$, $\pi:V'\to V'_p$ is a bijection, $0\notin V'_p$, and $0\notin V'_p+V'_p$. By property (iii), we have that $\dim_2^-(V')\leqslant d$ which certainly implies that $\dim_2^-(V'_p)\leqslant d$ (this second notion of $\dim_2^-$ is considered over $\mathbf{Z}/p\mathbf{Z}$). Hence, by Lemma \ref{lem:small-dim-rectification}, there exists a $\lambda\in(\mathbf{Z}/p\mathbf{Z})^\times$ such that $\lambda\cdot V'_p\subset [-q,q]\subset\mathbf{Z}/p\mathbf{Z}$, where 
    \begin{equation}\label{eq:qbound!}
    q\leqslant 8dp^{1-1/(2d)}.
    \end{equation}Lemma \ref{lem:dilating} further shows that, if $q<p/(2n)$, then the set $V''\vcentcolon=\pi^{-1}(\lambda \cdot V'_p)\subset [-q,q]\subset\mathbf{Z}$ satisfies
    $$\ML(V')\geqslant \ML(V'')-\frac{q}{p},$$ and has $\dim_2^-(V'')\leqslant \dim_2^-(V')\leqslant d$, so that $V''$ satisfies (iii). As $V'$ satisfies property (ii) by definition, \eqref{eq:MLclose} shows that \begin{equation}\label{eq:MLV''bound}
        \ML(V)\geqslant \ML(V')-\frac{1}{m^{\varepsilon}}\geqslant \ML(V'')-\frac{1}{m^{\varepsilon}}
    -\frac{q}{p}. 
    \end{equation}
    As we assumed that $m$ was minimal, we must either have that $q\geqslant p/(2n)$ (so that our application of Lemma \ref{lem:dilating} was invalid), that $q\geqslant m$ (so that $V''\subset [q]$ is not a `denser model' than $V'$), or else that $q$ and $V''$ do not satisfy one of the properties (i) or (ii).\footnote{There is a minor technicality here, namely that this set $V''$ can be a subset of $\mathbf{Z}$, rather than $\mathbf{N}$. This is never an issue, as we can without loss of generality replace all negative numbers $-v$ appearing in $V''$ by $v$ to make $V''$ a set of positive numbers, and clearly doing so does not affect $\ML(V'')$ as $\lVert tv\rVert_\mathbf{T}=\lVert -tv\rVert_\mathbf{T}$ for all $t$, and it also does not affect $\dim_2^-(V'')$. As we chose $p$ so that $0\notin V'_p+V'_p$, $V''$ never contains both $v$ and $-v$, so this procedure does not make $V''$ into a multiset.} So one of the following must hold:
    \begin{itemize}
        \item [(1)] $q\geqslant \min(p/(2n), m)$,
        \item [(2)] $q\leqslant e^{d(\log n)^{4}}$,
        \item [(3)] or $\frac{1}{m^\varepsilon}+\frac{q}{p}\geqslant \frac{1}{q^\varepsilon}$.
    \end{itemize}
    First assume that (2) is true, so $q\leqslant e^{d(\log n)^{4}}$. In this case, we cannot necessarily contradict that $m$ was the minimal integer satisfying (i),(ii),(iii) (as $q$ does not satisfy (i)), but we will show that $V''$ is nevertheless the desired `model' for the conclusion of Proposition \ref{prop:smalldimintro}. To see this, note that $V''\subset [T]$ with $T:=q\leqslant e^{d(\log n)^{4}}$ (as we are assuming that (2) holds), and by \eqref{eq:MLV''bound}, we get that \begin{align*}
        \ML(V)&\geqslant \ML(V'')-\frac{1}{m^\varepsilon}-\frac{q}{p}\\
        &\geqslant \ML(V'')-\frac{1}{e^{(\log n)^2}}-\frac{8d}{p^{1/(2d)}},
    \end{align*}
    using property (i) that $m>e^{d(\log n)^4}$ while $\varepsilon=\frac{1}{d(\log n)^2}$ to bound $1/m^{\varepsilon}$, and \eqref{eq:qbound!} to bound $q/p$. As we chose $p\in(2m,4m]$, property (i) also implies that $p>2m>e^{d(\log n)^4}$, showing that $\ML(V)\geqslant \ML(V'')-\frac{O(n)}{e^{(\log n)^2}}\geqslant \ML(V'')-O(\frac{1}{n^{100}})$, as desired.

    \medskip

    Now suppose that (1) holds. As we chose $p\in(2m,4m]$, (1) shows that $q\geqslant \min(\frac{p}{2n},m)\geqslant m/n$. Hence, as $q\leqslant 8dp^{1-1/(2d)}$ by \eqref{eq:qbound!} and as $p\leqslant 4m$, we get that $$m\leqslant nq\leqslant 8ndp^{1-1/(2d)}\leqslant8n^2(4m)^{1-1/(2d)},$$ using also the trivial bound $d=\dim_2(V)\leqslant n$. Rearranging this shows that $m\leqslant e^{d(\log n)^{O(1)}}$, so we see that $V'\subset [m]$ is the desired `model' for the conclusion of Proposition \ref{prop:smalldimintro}: note that $V'\subset [T]$ where $T:=m$ satisfies the bound $T\leqslant e^{d(\log n)^{O(1)}}$, and the fact that $V',m$ satisfy (i) and (ii) (by definition) immediately implies that $\ML(V')\geqslant \ML(V)-\frac{1}{e^{(\log n)^{2}}}\geqslant \ML(V')-O(n^{-100})$.

    \medskip
    
Finally, suppose that only (3) holds, so $\frac{1}{m^\varepsilon}+\frac{q}{p}\geqslant \frac{1}{q^\varepsilon}$. Then either $1/m^{\varepsilon}\geqslant 1/(eq)^{\varepsilon}$, or else $q/p\geqslant(1-1/e^{\varepsilon})/q^{\varepsilon}$. In the first case, we have that $q\geqslant m/e$ which recovers the condition in (1) (up to an irrelevant factor of $1/e$), so the same argument as in the previous paragraph works. In the second case, \eqref{eq:qbound!} shows that $$\frac{8d}{p^{1/(2d)}}\geqslant\frac{q}{p}\geqslant \frac{1-1/e^{\varepsilon}}{q^{\varepsilon}}\gg \frac{\varepsilon}{q^\varepsilon}=\frac{(d(\log n)^2)^{-1}}{q^{1/(d(\log n)^2)}},$$ using in the third equality that $e^{-\varepsilon}\leqslant 1-\varepsilon/2$ for $\varepsilon\in [0,1]$, and that $\varepsilon=1/(d(\log n)^2)$. As $q\leqslant p$ (else (1) holds), raising both sides to the power $2d$ shows that $d^{O(d)}/p\geqslant 1/((d\log n)^{O(d)}p^{2/(\log n)^2})$, so that $p\leqslant e^{d(\log n)^{O(1)}}$. Since we chose $p>2m$, this implies that $m\leqslant e^{d(\log n)^{O(1)}}$, which again implies that $V'\subset [m]$ is the desired `model' for the conclusion of Proposition \ref{prop:smalldimintro}: again, note that $V'\subset[T]$ where $T:=m$ satisfies the bound $T\leqslant e^{d(\log n)^{O(1)}}$, and the fact that $V',m$ satisfy (i) and (ii) (by definition) immediately implies that $\ML(V)\geqslant \ML(V')-\frac{1}{e^{(\log n)^{2}}}\geqslant \ML(V')-O(n^{-100})$.  
\end{proof}

Now we move on to proving Proposition \ref{prop:boundfordense}. We begin with some elementary results.
\begin{lemma}\label{lem:modpmeasure}
Let $V$ be a set of $n$ distinct positive integers. Suppose that $p$ is a prime, and that no element of $V$ is divisible by $p$. Then
\[\ML(V) \geqslant \frac{\left\lceil \frac{p-1}{2n}\right\rceil}{p}.\]
\end{lemma}
\begin{proof}
Let $\delta=\ML(V)$ and recall that $\Phi(x)=\sum_{v\in V}\phi(vx)$ where $\phi(x)=\mathbf{1}_{[-\delta,\delta]}(x)$, see definition \eqref{eq:Phidefi}. We apply Lemma \ref{lem:Phiprob} with the discrete measure $\mu=\frac{1}{p-1}\sum_{j=1}^{p-1}\delta_{j/p}$, where $\delta_{x_0}$ denotes the Dirac measure at $x_0$, i.e.~it is a unit point mass at $x_0$. Lemma \ref{lem:Phiprob} shows that
\begin{align*}
    1\leqslant \int_\To\Phi(x)\,d\mu(x)=\sum_{v\in V}\int_{\To}\phi(vx)\,d\mu(x).
\end{align*}
The assumption that no $v\in V$ is divisible by $p$ implies that as $j$ ranges over $(\mathbf{Z}/p\mathbf{Z})^\times$, so does $jv$. Hence, for each $v\in V$ we have that
\begin{align*}
\int_{\To}\phi(vx)\,d\mu(x)&=\frac{1}{p-1}\sum_{j=1}^{p-1}\phi(vj/p)\\
&=\frac{1}{p-1}\sum_{j=1}^{p-1}\phi(j/p)\\
&=\frac{\big|[-\delta,\delta]\cap\{j/p:j \in[-\frac{p-1}{2},\frac{p-1}{2}]\}\big|}{p-1}\\
&=\frac{2\lfloor p\delta\rfloor}{p-1}.
\end{align*}
Plugging this into the previous inequality shows that $\lfloor p\delta\rfloor\geqslant \frac{p-1}{2n}$. As $\lfloor p\delta\rfloor$ is an integer, we may replace the right hand side of this inequality by $\lceil \frac{p-1}{2n}\rceil$, so $p\delta\geqslant \lfloor p\delta\rfloor\geqslant \lceil \frac{p-1}{2n}\rceil$.
\end{proof}
%\BB{
The previous lemma essentially allows us to improve on the trivial bound $\ML(V)\geqslant \frac{1}{2n}$ by an additive term of size $\Omega( \frac{1}{p})$ if the prime $p$ does not divide any element of $V$, and the residue $p\md{2n}$ is not too close to $2n$ so that $\left\lceil \frac{p-1}{2n}\right\rceil \geqslant \frac{p-1}{2n}+\Omega(1)$. We make this precise in the following corollary. 
\begin{cor}\label{cor:modp}
    Let $V$ be a set of $n$ distinct positive integers. If $p$ is a prime such that no element of $V$ is divisible by $p$, and $p\equiv r\md{2n}$ for some $2\leqslant r\leqslant (1-c)2n$, then $$ \ML(V)\geqslant \frac{1}{2n}+\frac{c}{p}.$$
\end{cor}
\begin{proof}
    As $p-1\md{2n}=r-1\geqslant 1$, we find that
$\left\lceil\frac{p-1}{2n}\right\rceil = \frac{p-1}{2n}+\frac{2n-r+1}{2n}=\frac{p}{2n}+(1-\frac{r}{2n}),$
and hence Lemma \ref{lem:modpmeasure} shows that
$$\ML(V)\geqslant \frac{\left\lceil\frac{p-1}{2n}\right\rceil}{p}\geqslant \frac{1}{2n}+\frac{1}{p}\left(1-\frac{r}{2n}\right)\geqslant \frac{1}{2n}+\frac{c}{p},$$ since $r\leqslant (1-c)2n$.
\end{proof}
%}
With this corollary in hand, the proof of Proposition \ref{prop:boundfordense} is straightforward. 
\begin{proof}[Proof of Proposition \ref{prop:boundfordense}]
    Let $V$ be a set of $n$ distinct positive integers, and suppose that $V\subset [T]$ for some parameter $T\leqslant e^{n^{1+o(1)}}$. We claim that there is a prime $p\ll n \log T$ which satisfies
\begin{itemize}
    \item $p\nmid v$ for all $v\in V$,
    \item $p\md{2n}\in [2,(1-c)2n]$, for some absolute constant $c>0$.
\end{itemize}
Note that if such a prime $p$ exists, then Corollary \ref{cor:modp} yields the desired bound
$$\ML(V)\geqslant \frac{1}{2n}+\frac{c}{p}\geqslant \frac{1}{2n}+\Omega\left(\frac{1}{n\log T}\right),$$ so it remains to prove this claim. 

\medskip

Suppose that every prime $p\leqslant X$ for which $p\md{2n}\in [2,(1-c)2n]$ divides some $v\in V$. To prove the claim, it suffices to show that this is only possible if $X\ll n\log T$. This assumption implies that 
\begin{align}
\label{eq:prodp<T}
    \prod_{\substack{p\leqslant X\\p\mdsub{2n}\in [2,(1-c)2n]}}p\leqslant \prod_{v\in V}v\leqslant T^n.
\end{align}
Heuristically, the congruence condition $p\md{2n}\in[2,(1-c)2n]$ keeps at least $99.9\%$ of primes if we choose $c$ to be a sufficiently small constant, which then suggests that the left hand side is at least $\left(\prod_{p\leqslant X}p\right)^{\Omega(1)}$. The Prime Number Theorem shows that $\prod_{p\leqslant X}p=e^{\sum_{p\leqslant X}\log p}\geqslant e^{\Omega(X)}$, which would then combine with \eqref{eq:prodp<T} to show that $X\ll n\log T$, confirming the claim. There are various ways to make this rigorous. For example, it is a standard fact \cite[Theorem 10.5, Exercise 10.5.6]{iwanieckowalski} from analytic number theory that if $Y,y$ are parameters and $y\in[Y^{1/3},Y]$ (say), then almost all subintervals of length $y$ in $\{Y/2,\dots,Y+y\}$ contain $\Omega(y/\log Y)$ primes, i.e.~ that 
\begin{align}\label{eq:primesinaa}
    \#\left\{m\in[Y/2,Y]: \pi(m+y)-\pi(m)\gg \frac{y}{\log Y}\right\}=(1-o(1))\frac{Y}{2}.
\end{align} We will apply this with $Y=X$ and $y=n/100$. First note that replacing $X$ by $\min(X,n^{2.5})$ can only decrease $X$ and hence does not affect the validity of inequality \eqref{eq:prodp<T}, and note that doing so also does not affect whether or not the desired claim $X\ll n\log T$ holds: the bounds $X\ll n\log T$ and $\min(X,n^{2.5})\ll n\log T$ are equivalent since $T\leqslant e^{n^{1+o(1)}}$. Hence, may without loss of generality assume that $X\leqslant n^{2.5}$, so that $y=n/100$ satisfies the assumption that $y>Y^{1/3}$ for $Y:=X$. We see that
\begin{align*}
    &\#\left\{m\in[X/2,X]: m\md{2n}\in[2,1.9n]\text{ and } \pi(m+n/100)-\pi(m)\gg \frac{n}{\log X}\right\}
    \\&\geqslant \#\left\{m\in[X/2,X]: \pi(m+n/100)-\pi(m)\gg \frac{n}{\log X}\right\}-X/10\\
    &\geqslant (1-o(1))\frac{X}{2}-X/10\geqslant X/3.
\end{align*}
by using \eqref{eq:primesinaa} for the final inequality. So there are at least $\frac{X}{3}\times \Omega(n/\log X)$ pairs $(m,p)$ with $m\in[X/2,X]$ satisfying $m\md{2n} \in[2,1.9n]$, and $p$ being a prime that lies in the interval $[m,m+n/100]$. Each prime $p$ trivially occurs in at most $n/100$ pairs $(m,p)$, so we find that the interval $[X/2,X]$ contains $\Omega(X/\log X)$ primes $p$ which satisfy: $p\md{2n}\in [2,1.9n]+[0,n/100]\subset [2,1.99n]$. Hence,
\begin{align*}
    \prod_{\substack{p\leqslant X\\p\mdsub{2n}\in[2,1.99n]}}p\geqslant \left(\frac{X}{2}\right)^{\Omega(X/\log X)}\geqslant e^{\Omega(X)}.
\end{align*}
Plugging this into \eqref{eq:prodp<T} gives the bound $X\ll n\log T$, confirming the claim.
\end{proof}
Recall that Tao \cite[Proposition 1.6]{Tao2018LonelyRunner} showed that if $V\subset[Cn]$ for an absolute constant $C$, then $\ML(V)\geqslant (1+\Omega_C(1))\frac{1}{2n}$, giving an improvement  by a constant factor over the trivial bound $\frac{1}{2n}$ for sets $V$ which have positive density in an initial interval of $\mathbf{N}$. We finish by noting that our approach gives a simple way of recovering this result (with a better dependence on $C$). By the Prime Number theorem, we may pick a prime $p\in (Cn, Cn+2n]$ with $p\md{2n}\in [2,n]$ (say). As $p>Cn\geqslant\max( V)$, $p$ trivially does not divide any element of $V$, so applying Corollary \ref{cor:modp} with $c=1/2$ shows that $$\ML(V)\geqslant\frac{1}{2n}+\frac{1/2}{p}\geqslant\left(1+\Omega\left(\frac{1}{C}\right)\right)\frac{1}{2n}.$$

\bibliographystyle{plain}
\bibliography{referencesLonelyRunner}
\clearpage
\appendix
\section*{Appendix}\label{sec:appendix}
In this appendix, we provide a short combinatorial proof of the following lemma, giving the following slightly weaker (but still sufficient) bound in \eqref{eq:E_kP2} compared to Lemma \ref{lem:PE_kbound}.
\begin{lem*}
    Let $A\subset\mathbf{Z}$ be a $2$-dissociated set, and define \begin{equation*}
        E_k=\left\{\sum_{a\in A}\varepsilon_aa:\varepsilon\in\{-1,0,1\} \text{ and }\sum_a|\varepsilon_a|=k\right\}.
    \end{equation*} Then any arithmetic progression $P$ contains at most $|E_k\cap P|\leqslant (C_1(\log |P|)(\log |A|))^{k}$ elements of $E_k$,
    where $C_1$ is an absolute constant.
\end{lem*}
Our argument is based on the recent progress in \cite{bell} on the sunflower problem of Erd\H{o}s and Rado. A \emph{sunflower} of size $r$ is a collection of sets $U_1,U_2,\dots,U_r$ such that all their pairwise intersections $U_i\cap U_j$ with $i\neq j$ are equal. This is equivalent to the existence of a set $X$, called the \emph{kernel} of the sunflower, which is a subset of each $U_i$ and such that the sets $U_1\setminus X,\dots,U_r\setminus X$ are pairwise disjoint. Let $S_1,\dots,S_\ell$ be any collection of sets of size $k$. The main result of \cite{bell} states that $\{S_1,\dots,S_\ell\}$ either contains a sunflower of size $r$, or else that 
\begin{equation}\label{eq:bellsunflower}
\ell \leqslant (C_2r\log k)^k,
\end{equation}
where $C_2$ is an absolute constant that works for all $k,\ell$.
\begin{proof}[Proof of \eqref{eq:E_kP2}]
    Let $A\subset\mathbf{Z}$ be $2$-dissociated, and define $E_k$ as in \eqref{eq:AE_kdefi}. The case where $k=1$ is easy: if $B= A\cap P$, then the $2^{|B|}$ distinct sums $\sum_{b\in S}b$, with $S$ ranging over all subsets of $B$, are contained in $\underbrace{P+\dots+P}_{|B|}$, which is itself an arithmetic progression of size at most $|B||P|$. As $A$ is $2$-dissociated (in fact, $1$-dissociativity suffices here), these $2^{|B|}$ subset sums are distinct so that $2^{|B|} \leqslant |B| |P|$. This shows that $|B| \ll \log |P|$.

    \medskip

    Now suppose that $k\in [2,|A|]$ and that $B=P\cap E_k$. By definition of $E_k$, each $b\in B$ is a sum of $k$ elements of $E_1=A\cup-A$: \begin{equation}\label{eq:bexpansion}
        b=\sum_{z\in S_b}z,
    \end{equation} for some subsets $S_b\subset E_1$ of size $k$. As $A$ is $2$-dissociated, these `support' sets $S_b$ are uniquely determined. Now we apply the sunflower result \eqref{eq:bellsunflower} to the collection of sets $\{S_b:b\in B\}$, noting that all $S_b$ have size $k$, and with $r=C_3\log |P|$ for some absolute constant $C_3$ to be chosen later. Hence, we either find that $$|B|\leqslant (C_2C_3(\log|P|)(\log k))^k\leqslant (C_2C_3(\log|P|)(\log |A|))^k$$ in which case we are done, or else we can find a sunflower $(S_{b})_{b\in B'}$ for some $B'\subset B$ of size $|B'|=r=C_3\log |P|$. So it suffices to show that the second option is impossible, for some sufficiently large fixed constant $C_3$. Suppose for a contradiction that the second option were true, so that all the sets $S_{b}$ with ${b\in B'}$ contain a fixed kernel $X$, but are otherwise disjoint. Defining $t:=\sum_{x\in X}x$, we can rewrite \eqref{eq:bexpansion} as 
    \begin{align}\label{eq:sunexpansion}\sum_{z\in S_b\setminus X}z=b-t\in P-t\end{align} for every $b\in B'$. As the sets $(S_b)_{b\in B'}$ form a sunflower with kernel $X$, the numbers $b-t$ with $b\in B'$ have pairwise disjoint `supports' $S_b\setminus X$, and they all lie in $P-t$, which is itself an arithmetic progression of size $|P|$. We will prove the following claim.
    \begin{claim*}
        There exists a subset $B''\subset B'$ of size $|B''|\geqslant |B'|/2$ for which the set $\{b-t:b\in B''\}$ is $1$-dissociated.
    \end{claim*} Note that as $\{b-t:b\in B''\}$ is then a dissociated subset of $P-t$, repeating the argument in the first paragraph (for $k=1)$ shows that $|B''|\ll \log |P-t|=\log |P|$. On the other hand, $|B''|\geqslant |B'|/2=\frac{C_3}2 \log |P|$, and hence this indeed gives the required contradiction upon choosing the constant $C_3$ to be sufficiently large.
    
    \medskip
    
    It only remains to prove the claim. Recall that the sets $S_b\subset A\cup-A$ with $b\in B'$ form a sunflower with kernel $X$, so that the sets $S_b\setminus X$ with $b\in B'$ are pairwise disjoint. We describe an algorithmic procedure for finding $B''$, which we initialise by taking $B''=\emptyset$ and $C=B'$. While $|C|\geqslant 1$, we choose an arbitrary element $b\in C\subset B'$ and add it to $B''$. The sunflower condition shows that $S_{b}\setminus X$ is disjoint from $\bigcup_{c\in C\setminus\{b\}}(S_{c}\setminus X)$. 
    \par \textbf{Case I:} If $-(S_{b}\setminus X)$ is also disjoint from $\bigcup_{c\in C\setminus\{b\}}(S_{c}\setminus X)$, then we simply remove $b$ from $C$ and go back to the first step (selecting again an arbitrary next element from the new set $C$). 
    \par \textbf{Case II:} Else, there exists a $z\in A\cup-A$ such that $z\in S_{b}\setminus X$ and $-z\in (S_{b'}\setminus X)$, for some unique $b'\in C\subset B'$. There may be multiple such $z\in A\cup-A$, in which case we choose one arbitrarily, but note that for each $z$ the corresponding element $b'$ is unique (if it exists) since the sets $S_{b'}\setminus X$ with $b'\in B'$ are pairwise disjoint. In this case, we remove both $b$ and $b'$ from $C$, and then go back to the first step.
    
    \medskip
    
    The procedure starts with $C=B'$ and continues until $C$ is empty. At each step, we add a new element $b$ from $C$ to $B''$ while, in both cases I and II, we remove $b$ and at most one other element from $C$. So clearly $|B''|\geqslant |B'|/2$. To see why the set $\{b-t:b\in B''\}$ is $1$-dissociated, suppose that $$\sum_{b\in B''}\varepsilon_b (b-t)=0,$$ for some $\varepsilon_b\in\{-1,0,1\}$. Then by \eqref{eq:sunexpansion}, we see that \begin{align}\label{eq:B''diss}\sum_{b\in B''}\varepsilon_b\sum_{z\in S_b\setminus X} z=0.
    \end{align}The sets $S_b\setminus X$ with $b\in B''$ are pairwise disjoint subsets of $A\cup -A$, so the relation above produces a relation between elements $a$ of $A$ with coefficients $\eta_a\in \{-2,-1,0,1,2\}$. Note that even though $\varepsilon_b\in\{-1,0,1\}$ for all $b\in B''$, it is still possible that $\eta_a\in \{-2,2\}$ for some $a\in A$ because of the possibility that $a\in S_{b_1}\setminus X$ and $-a\in S_{b_2}\setminus X$, for two distinct $b_1,b_2\in B''$. As $A$ is $2$-dissociated, this implies that all $\eta_a$ are $0$, and we will show that this forces all $\varepsilon_b$ to vanish too. Let us write $B''=\{b_1,b_2,\dots,b_{|B''|}\}$, ordered such that $b_1$ was chosen first in the algorithm above, then $b_2$, and so on. By the way we have constructed $B''$, for every $b_j\in B''$ there exists a $z_j\in S_{b_j}\setminus X$ such that $$z_j\notin \bigcup_{i>j}((S_{b_i}\setminus X)\cup-(S_{b_i}\setminus X)).$$ So in particular, $z_1\in S_{b_1}\setminus X$ but $z_1$ does not lie in $\pm(S_b\setminus X)$ for any other $b\in B''$, so it is clear from \eqref{eq:B''diss} that $\eta_{z_1}=\varepsilon_{b_1}$, and hence $\varepsilon_{b_1}=0$. Next, given that $\varepsilon_{b_1}=0$ and that $z_2\in S_{b_2}\setminus X$ while $z_2$ does not lie in $\pm(S_{b_i}\setminus X)$ for any $i>2$, we similarly deduce from \eqref{eq:B''diss} that $\eta_{z_2}=\varepsilon_{b_2}$, so that $\varepsilon_{b_2}=0$. Continuing in this way, we see that $\varepsilon_b=0$ for all $b\in B''$, confirming our claim that $\{b-t:b\in B''\}$ is $1$-dissociated.
\end{proof}

\bigskip

\noindent
{\sc Department of Pure Mathematics \& Mathematical Statistics, Centre for Mathematical Sciences, Wilberforce Road, Cambridge, CB3 0WB, UK.}\newline
\href{mailto:bb741@cam.ac.uk}{\small bb741@cam.ac.uk}

\end{document}